\documentclass[11pt]{amsart}
\usepackage{amssymb,latexsym}
\usepackage{enumerate}

\makeatletter
\@namedef{subjclassname@2010}{%
  \textup{2010} Mathematics Subject Classification}
\makeatother

\newtheorem{theorem}{Theorem}[section]
\newtheorem{lemma}[theorem]{Lemma}
\newtheorem{claim}[theorem]{Claim}
\newtheorem{example}[theorem]{Example}
\newtheorem{problem}[theorem]{Problem}

\numberwithin{equation}{section}

\newcommand{\w}{\omega}
\newcommand{\e}{\varepsilon}

\newcommand{\IN}{\mathbb N}
\newcommand{\IR}{\mathbb R}
\newcommand{\pr}{\mathrm{pr}}
\newcommand{\la}{\langle}
\newcommand{\ra}{\rangle}
\newcommand{\supp}{\mathrm{supp}}
\newcommand{\conv}{\mathrm{conv}}
\newcommand{\aff}{\mathrm{aff}}
\newcommand{\lin}{\mathrm{lin}}
\newcommand{\Ra}{\Rightarrow}
\newcommand{\Ker}{\mathrm{Ker}}
\newcommand{\A}{\mathcal A}
\newcommand{\C}{\mathcal C}

\begin{document}





\title[A characterization of polyhedral convex sets]{A ``hidden'' characterization of\\ polyhedral convex sets}

\author[T. Banakh]{Taras Banakh}
\address{Ivan Franko National University of Lviv, Ukraine, and\newline\indent Jan Kochanowski University, Kielce, Poland}
\email{t.o.banakh@gmail.com}

\author[I. Hetman]{Ivan Hetman}
\address{Ivan Franko National University of Lviv\newline\indent
Universytetska 1, Lviv, 79000, Ukraine}
\email{ihromant@gmail.com}

\date{}

\begin{abstract}
We prove that a closed convex subset $C$ of a complete linear metric space $X$ is polyhedral in its closed linear hull if and only if no infinite subset $A\subset X\backslash C$ can be hidden behind $C$ in the sense that $[x,y]\cap C\not = \emptyset$ for any  distinct points $x,y\in A$.
\end{abstract}

\subjclass[2010]{Primary 46A55, 52A07; Secondary 52B05, 52A37}

\keywords{Polyhedral convex set; a hidden set; complete linear metric space}

\maketitle

\section{Introduction}
A convex subset $C$ of a real linear topological space $L$ is called \mbox{\em polyhedral} {\em in} $L$ if it can be written as a finite intersection $C=\bigcap_{i=1}^nf_i^{-1}\big((-\infty,a_i]\big)$ of closed half-spaces determined by some linear continuous functionals $f_1,\dots,f_n:L\to\IR$ and some real numbers $a_1,\dots,a_n$, see \cite{A}.

This notion has also an algebraic version. We shall say that a convex subset $C$ of a linear space $L$ is {\em polyhedric} in a convex set $D\supset C$ of $L$ if
$C=\bigcap_{i=1}^nH_i$ for some convex subsets $H_1,\dots,H_n\subset D$ having convex complements $D\setminus H_i$, $i\le n$.

In this paper polyhedral sets will be characterized with help of a combinatorial notion of a hidden set.

We say that a subset $A$ of a linear space $L$ is {\em hidden behind
a set $C\subset L$} if $A\subset L\setminus C$ and for any distinct
points $a,b\in A$ the closed segment $[a,b]=\{ta+(1-t)b:t\in[0,1]\}$
meets the set $C$. In this case we shall also say that the set $C$
{\em hides} the set $A$.

The main result of this paper is the following ``hidden'' characterization of closed polyhedral convex sets in complete linear metric spaces. This characterization has been applied in the paper \cite{BH2} devoted to a characterization of approximatively polyhedral convex sets in Banach spaces, which was applied in the paper \cite{BHS} devoted to recognizing the topological type of connected components of the hyperspace of closed convex subsets of a Banach space. Some other characterizations of polyhedral convex sets can be found in \cite{WW} and \cite{Klee}.

\begin{theorem}\label{main} For a closed convex subset $C$ of a complete linear metric space $X$ the following conditions are equivalent:
\begin{enumerate}
\item $C$ is polyhedral in its closed linear hull\/ $\overline{\lin}(C)$;
\item $C$ is polyhedric in its affine hull $\aff(C)$;
\item $C$ hides no infinite subset $A\subset X\setminus C$.
\end{enumerate}
\end{theorem}

The proof of this theorem is rather long and will be presented in Section~\ref{main:pf}. Now let us show that the assumption of the completeness of the linear space $X$ in Theorem~\ref{main} is essential. A suitable  counterexample will be constructed in the (non-complete) normed space
$$c_{00}=\{(x_n)_{n\in\w}\in \IR^\w:\exists n\in\w\;\forall m\ge n\;\;x_m=0\}$$endowed with $\sup$-norm
$\|x\|=\sup_{n\in\w}|x_n|$, where $x=(x_n)_{n\in\w}\in c_{00}$.

\begin{example} The standard infinite-dimensional simplex
$$\Delta=\{(x_n)_{n\in\w}\in c_{00}\cap[0,1]^\w:\sum_{n\in\w}x_n=1\}\subset c_{00}$$ hides no infinite subset of $c_{00}\setminus\Delta$ but is not polyhedral in $c_{00}$.
\end{example}

\begin{proof} First we show that the simplex $\Delta$ is not polyhedral in $c_{00}$. Assuming the opposite, we would find  linear functionals $f_1,\dots,f_n:c_{00}\to\IR$ and real numbers $a_1,\dots,a_n$ such that $\Delta=\bigcap_{i=1}^nf_i^{-1}((-\infty,a_i])$. Consider the linear subspace $X_0=\bigcap_{i=1}^n f_i^{-1}(0)$ that has finite codimension in $c_{00}$. It follows that for each $x_0\in\Delta$, we get $x_0+X_0\subset\bigcap_{i=1}^nf_i^{-1}((-\infty,a_i])=\Delta$, which implies that the set $\Delta$ is unbounded. This contradiction shows that $\Delta$ is not polyhedral in $c_{00}$.

Now assume that some infinite subset $A\subset c_{00}\setminus\Delta$
can be hidden behind the simplex $\Delta$. Decompose the space
$c_{00}$ into the union $c_{00}=\Sigma_<\cup\Sigma_1\cup\Sigma_>$ of the
sets
$$
\begin{aligned}
\Sigma_<&=\{(x_n)_{n\in\w}\in c_{00}:\sum_{n\in\w}x_n<1\},\\
\Sigma_1&=\{(x_n)_{n\in\w}\in c_{00}:\sum_{n\in\w}x_n=1\},\\
\Sigma_>&=\{(x_n)_{n\in\w}\in c_{00}:\sum_{n\in\w}x_n>1\}.
\end{aligned}
$$
Observe that for any two points $x,y\in\Sigma_{<}$ the segment $[x,y]$ does not intersect $\Delta$. Consequently, $|A\cap\Sigma_{<}|\le 1$. By the same reason, $|A\cap\Sigma_>|\le 1$. So, we lose no generality assuming that $A\subset\Sigma_1\setminus\Delta$. For each element $a\in A$ let
$$\supp_+(a)=\{n\in\w:x_n>0\}\mbox{ \ and \ }\supp_-(a)=\{n\in\w:x_n<0\}.$$
It is easy to see that each point $a\in\Sigma_1\setminus\Delta$ has non-empty negative support $\supp_-(a)$.

Fix any point $b\in A$. We claim that for any point $a\in A\setminus\{b\}$ we get $\supp_-(a)\subset\supp_+(b)$.
In the opposite case the set $\supp_-(a)\setminus \supp_+(b)$ contains some number $k\in\w$ and then
$[a,b[\subset \{(x_n)_{n\in\w}\in c_{00}:x_k<0\}\setminus\Delta$ which is not possible as $\{a,b\}$ is hidden behind $\Delta$. Since (the power-set of) the set $\supp_+(b)$ is finite and $A\setminus\{b\}$ is infinite, the Pigeonhole Principle yields two distinct points $a,a'\in A$ such that $\supp_-(a)=\supp_-(a')\subset\supp_+(b)$. Now we see that for any $k\in\supp_-(a)=\supp_-(a')$, we get $[a,a']\subset\{(x_n)_{n\in\w}\in c_{00}:x_n<0\}\subset c_{00}\setminus\Delta$, which contradicts the choice of $A$ as a set hidden behind $\Delta$.
\end{proof}

\section{Preliminaries}

In this section we prove some lemmas which will be used in the proof of Theorem~\ref{main}.

\begin{lemma}\label{l2.1} Let $T:X\to Y$ be a linear continuous operator between linear topological spaces. If a convex subset $D\subset Y$ is polyhedral in its closed linear hull $\overline{\lin}(D)$, then its preimage $C=T^{-1}(D)$ is polyhedral in its closed linear hull $\overline{\lin}(C)$.
\end{lemma}

\begin{proof} Write the polyhedral set $D$ as a finite intersection $$D=\bigcap_{i=1}^nf_i^{-1}((-\infty,a_i])$$ of closed half-spaces defined by linear continuous functionals $f_1,\dots,f_n:\overline{\lin}(D)\to\IR$ and real numbers $a_1,\dots,a_n$.   The continuity of the operator $T$ implies that $T(\overline{\lin}(C))\subset\overline{\lin}(D)$. Consequently, for every $i\le n$ the linear continuous functional $g_i=f_i\circ T:\overline{\lin}(C)\to\IR$ is well-defined. Since $C=T^{-1}(D)=\bigcap_{i=1}^ng_i^{-1}((-\infty,a_i])$, the set $C$ is polyhedral in $\overline{\lin}(C)$.
\end{proof}

 An operator $A:X\to Y$ between two linear spaces is called {\em affine} if $$A(tx+(1-t)y)=tA(x)+(1-t)A(y)\mbox{ \ for any $x,y\in X$ and $t\in\IR$.}$$ It is well-known that an operator $A:X\to Y$ is affine if and only if the operator $B:X\to Y$, $B:x\mapsto A(x)-A(0)$, is linear. The following lemma trivially follows from the definition of a hidden set.

\begin{lemma}\label{l2.2} Let $T:X\to Y$ be an affine operator between linear topological spaces, $D\subset Y$ be a convex set, $C=T^{-1}(D)$, and $A\subset X\setminus C$ be a subset such that $T|A$ is injective. The set $C$ hides the set $A$ if and only if the set $D=T(C)$ hides the set $T(A)$.
\end{lemma}

Let us recall that a convex subset $C$ of a linear topological space $X$ is called a {\em convex body} in $X$ if $C$ has non-empty interior in $X$.

\begin{lemma}\label{l2.3} Let $C$ be an infinite-dimensional closed convex subset of a complete linear metric space $Y$. If $C$ is infinite-dimensional, then there is an injective continuous affine operator $T:l_2\to Y$ such that $T^{-1}(C)$ is a closed convex body in $l_2$.
\end{lemma}

\begin{proof} By \cite[1.2.2]{Rol}, the topology of $Y$ is generated by a complete invariant metric $d$ such that the $F$-norm
$\|y\|=d(y,0)$ has the property $\|ty\|\le\|y\|$ for all $y\in Y$ and $t\in[-1,1]$.

 We lose no generality assuming that the convex set $C$ contains the neutral element $0$ of $Y$. In this case for any points $y_n\in C$, $n\in\w$, and any non-negative real numbers $t_n$, $n\in\w$, with $\sum_{n\in\w}t_n\le 1$ we get $\sum_{n\in\w}t_ny_n\in C$ whenever the series $\sum_{n=0}^\infty t_ny_n$ converges in $Y$.

The set $C$ is infinite dimensional and hence contains a linearly independent sequence $(y_n)_{n=1}^\infty$. Multiplying each $y_n$ by a small positive real number, we can additionally assume that $\|y_n\|\le 2^{-n}$. It follows that series $\sum_{n=1}^\infty \frac1{4^n}y_n$ converges in $Y$ and its sum $s_0=\sum_{n=1}^\infty \frac1{4^n}y_n$ belongs to the closed convex set $C$ as $\sum_{n=1}^\infty\frac1{4^n}=\frac13\le 1$.

Let $l_2^f$ be the linear hull of the standard orthonormal basis $(e_n)_{n\in\w}$ in the separable Hilbert space $l_2$.
Define a linear operator $S:l_2\to Y$ letting $S(e_n)=\frac1{4^n}y_n$ for every $n\in\IN$. The convergence of the series $\sum_{n=1}^\infty\|y_n\|$ implies that the operator $S$ is continuous and hence can be extended to a continuous linear operator $\bar S:l_2\to Y$. Let $B_1=\{x\in l_2:\|x\|<1\}$ denote the open unit ball in the Hilbert space $l_2$. We claim that $\bar S(B_1)+s_0\subset C$. Indeed, for every $x=(x_n)_{n=1}^\infty\in B_1$ and every $n\in\IN$ we get $|x_n|\le 1$ and hence
$$\frac1{4^n}+\frac{x_n}{4^n}\ge \frac1{4^n}-\frac1{4^n}=0.$$ Taking into account that
$$\sum_{n=1}^\infty\big(\frac1{4^n}+\frac{x_n}{4^n}\big)\le\sum_{n=1}^\infty\frac{2}{4^n}=\frac23\le 1$$ and $0\in C$, we conclude that
$$s_0+\bar S(x)=\sum_{n=1}^\infty\frac1{4^n}y_n+\sum_{n=1}^\infty \frac{x_n}{4^n}y_n\in C.$$

Let $H=\bar S^{-1}(0)$ be the kernel of the operator $\bar S$ and $H^\perp$ be the orthogonal complement of $H$ in the Hilbert space $l_2$. It follows that the affine operator $T:H^\perp\to Y$, $T:x\mapsto \bar S(x)+s_0$, is injective and the preimage $T^{-1}(C)$ contains the unit ball $B_1\cap H^\perp$ of the Hilbert space $H^\perp$. So, $T^{-1}(C)$ is a closed convex body in the Hilbert space $H^\perp$. Since $\bar S(l_2)=\bar S(K^\perp)\supset\{y_n\}_{n\in\w}$, the Hilbert space $H^\perp$ is infinite-dimensional and hence can be identified with $l_2$.
\end{proof}

The following lemma is the most important and technically difficult
ingredient of the proof of Theorem~\ref{main}.

\begin{lemma}\label{l0} If a closed convex body $\bar C$ in a separable Hilbert space $X$ is not polyhedral, then $\bar C$ hides some infinite subset $A\subset X\setminus \bar C$.
\end{lemma}

\begin{proof} Let $\la \cdot,\cdot\ra$ denote the inner product of the Hilbert space $X$. Each element $y\in X$ determines a functional $y^*:x\mapsto\la x,y\ra$ on $X$. By the Riesz Representation Theorem \cite[3.4]{Conway}, the operator $y\mapsto y^*$ is a linear isometry between the Hilbert space $X$ and its dual Hilbert space $X^*$. By $S=\{x\in X:\|x\|=1\}$ and $S^*=\{x^*\in X^*:\|x^*\|=1\}$ we denote the unit spheres in  the Hilbert spaces $X$ and $X^*$, respectively.

Let $C$ be the interior of the convex body $\bar C$ and $\partial C=\bar C\setminus C$ be the boundary of $\bar C$ in the Hilbert space $X$. A functional $x^*\in S^*$ is called {\em supporting to $\bar C$} at a point $x\in \partial C$ if $x^*(x)=\sup x^*(C)$. The Hahn-Banach Theorem guarantees that for each point $x\in\partial C$ there is a supporting functional $x^*\in S^*$ to $C$ at $x$. If such a supporting functional is unique, then the point $x$ is called {\em smooth}.
By the classical Mazur's Theorem \cite[1.20]{Phelps}, the set $\Sigma$ of smooth points is a dense $G_\delta$ in the boundary $\partial C$ of $C$. By $\sigma:\Sigma\to S^*$ we shall denote the function assigning to each smooth point $x\in\Sigma$ the unique support functional $\sigma_x\in S^*$ to $C$ at $x$. Let us observe that the function $\sigma$ has closed graph $$\Gamma:=\{(x,\sigma_x):x\in\Sigma\}=(\Sigma\times S^*)\cap\{(x,x^*)\in C\times S^*:x^*(x)\ge\sup x^*(C)\}$$in the Polish space $\Sigma\times S^*$. Let $\pr_1:\Gamma\to \Sigma$ and $\pr_2:\Gamma\to S^*$ be the projections on the corresponding factors. Observe that the projection $\pr_1$ is a bijective and continuous map between Polish spaces. By the Lusin-Souslin Theorem \cite[15.1]{Ke}, it is a Borel isomorphism, which implies that the map $\sigma=\pr_2\circ \pr_1^{-1}:\Sigma\to S^*$ is Borel measurable.
By Theorem 8.38 \cite{Ke}, there is a dense $G_\delta$-subset $G\subset \Sigma$ such that the restriction $\sigma|G$ is continuous.

\begin{claim}\label{cl1} The image $\sigma(G)$ is infinite.
\end{claim}

\begin{proof} Assume that $\sigma(G)$ is finite and find functionals $f_1,\dots,f_n\in S^*$ such that $\sigma(G)=\{f_1,\dots,f_n\}$. Since the set $\bar C\subsetneq\bigcap_{i=1}^nf_i^{-1}((-\infty,\max f_i(\bar C)])$ is not polyhedral, there is a point $x\in X\setminus \bar C$ such that $f_i(x)\le \max f_i(\bar C)$ for all $i\le n$.
Fix any point $x_0\in C$. Since $C$ is open, $f_i(x_0)<\max f_i(\bar C)$ for all $i\le n$. Since $x\notin \bar C$, the segment $[x,x_0]$ meets the boundary $\partial C$ at some point $y=(1-t)x+tx_0$ where $t\in(0,1)$.
For the point $y$ we get $f_i(y)=(1-t)f_i(x)+tf_i(x_0)<\max f_i(\bar C)$. It follows that the set $U=\bigcap_{i=1}^n f_i^{-1}(-\infty,\max f_i(\bar C))$ is an open neighborhood of the point $y$ in $X$. Since the set $G$ in dense in $\partial C$, there is a point $z\in G\cap U$. Consider the unique supporting functional $\sigma_z$ to $C$ at the point $z$. The inclusion $z\in U$ implies that $\sigma_z\in \sigma(G)\setminus\{f_1,\dots,f_n\}$, which is a desired contradiction.
\end{proof}

Depending on the cardinality of the set $\sigma(G)$ we divide the further proof of Lemma~\ref{l0} into two Lemmas~\ref{l2.6} and \ref{l2}.

\begin{lemma}\label{l2.6} If the image $\sigma(G)$ is uncountable, then the set $C$ hides some infinite subset of $X$.
\end{lemma}

\begin{proof} The continuous map $\sigma|G$ induces a closed equivalence relation $E=\{(x,y)\in G\times G:\sigma(x)=\sigma(y)\}$ on the Polish space $G$. Since this equivalence relation has uncountably many equivalence classes, the Silver's Theorem \cite{Silver} yields a topological copy $K\subset G$ of the Cantor cube $\{0,1\}^\w$ such that $K$ has  at most one-point intersection with each equivalence class. This is equivalent to saying that the restriction $\sigma|K$ is injective. The existence of such Cantor set $K$ can be also derived from Feng's Theorem \cite{Feng} saying that the Open Coloring Axiom holds for analytic spaces.

For any $x\in K$ let $y_x\in S$ be the unique vector  such that $\sigma_x(z)=\la z,y_x\ra$ for all $z\in X$. For a real number $\e\in[0,1]$ consider the open subset $\Lambda(x,\e)=\{z\in K:[x+\e y_x,z]\cap C\ne\emptyset\}$ of the Cantor set $K$.

\begin{claim}\label{cl2} For any $x\in K$ the sets $\Lambda(x,\e)$ have the following properties:
\begin{enumerate}
\item $\Lambda(x,\e)\supset\Lambda(x,\delta)$ for any $0<\e\le \delta\le 1$;
\item $\bigcup_{\e\in(0,1]}\Lambda(x,\e)=K\setminus\{x\}$.
\end{enumerate}
\end{claim}

\begin{proof} 1. Fix any numbers $0<\e\le\delta\le 1$ and a point $z\in\Lambda(x,\delta)$. By the definition of the set $\Lambda(x,\delta)$, the segment $[x+\delta y_x,z]$ meets the open convex set  $C$ at some point $c$. Since the points $x,z$ belong to the convex set $\bar C$ and the point $c$ belongs to its interior $C$, the triangle $$\Delta=\{t_cc+t_xx+t_zz:t_c>0,\;t_x,t_z\ge 0,\;\;t_c+t_x+t_y=1\}$$ lies in the interior $C$ of $\bar C$. Since the segment $[x+\e y_x,z]$ intersects this triangle, it has non-empty intersection with $C$.
\smallskip

2. Take any point $z\in K\setminus\{x\}$. Since $\sigma|K$ is injective, the supporting functionals $\sigma_x$ and $\sigma_y$ are distinct. Then the open segment ${]x,z[}=[x,z]\setminus\{x,z\}$ lies in the interior $C$ of $\bar C$. In the opposite case, $[x,z]\subset\partial C$ and for the midpoint $\frac12x+\frac12z$ there would exist a supporting functional $x^*$, which would be supporting for each point of the segment $[x,y]$. This is impossible as the points $x,z$ are smooth and have unique and distinct supporting functionals. This contradiction proves that the segment $[x,z]$ meets the interior $C$ of $\bar C$. Then for some positive $\e>0$ the segment $[x+\e y_x,z]$ also meets $C$, which implies that $z\in\Lambda(x,\e)$.
\end{proof}

Being homeomorphic to the Cantor cube, the space $K$ carries an atomless $\sigma$-additive Borel probability measure $\mu$. Fix any point $x_0\in K$. Using Claim~\ref{cl2}(2), find $\e_0\in(0,1]$ such that $\mu(\Lambda(x_0,\e_0))>1-2^{-1}$. Next proceed by induction and construct a sequence of points $(x_n)_{n\in\w}$ and a sequence of positive real numbers $(\e_n)_{n\in\w}$ such that for every $n\in\IN$
\begin{enumerate}
\item $x_n\in \bigcap_{k<n}\Lambda(x_k,\e_k)$;
\item $\mu(\Lambda(x_n,\e_n))>1-2^{-n-1}$;
\item $[x_k+\e_k y_{x_k},x_n+\e_n y_{x_n}]\cap C\ne\emptyset$ for all $k<n$;
\item $x_n+\e_n y_{x_n}\notin \bar C$.
\end{enumerate}
Assume that for some $n$, the points $x_k$, $k<n$, and real numbers $\e_k$, $k<n$, have been
constructed. Consider the intersection $\bigcap_{k<n}\Lambda(x_k,\e_k)$ and observe that it has positive measure:
$$
\begin{aligned}
&\mu\big(\bigcap_{k<n}\Lambda(x_k,\e_k)\big)=1-\mu\big(K\setminus \bigcap_{k<n}\Lambda(x_k,\e_k)\big)=\\
&=1-\mu\big(\bigcup_{k<n}K\setminus\Lambda(x_k,\e_k)\big)\ge 1-\sum_{k<n}\mu(K\setminus\Lambda(x_k,\e_k))=\\
&=1-\sum_{k<n}
(1-\mu(\Lambda(x_k,\e_k))>1-\sum_{k<n}2^{-k-1}>0.
\end{aligned}
$$
So, this intersection is not empty and we can select a point $x_n$ satisfying the condition (1). For every $k<n$ the definition of the set $\Lambda(x_k,\e_k)$ ensures that the segment $[x_k+\e_k y_{x_k},x_n]$ meets the interior $C$ of the convex set $\bar C$. Consequently, there is $\e'_n>0$ such that for every $\e_n\le \e_n'$ and every $k<n$ the segment $[x_k+\e_k y_{x_k},x_n+\e_n y_{x_n}]$ still meets the open set $C$. Finally, using Claim~\ref{cl2}(2), choose a positive real $\e_n\in(0,\e_n']$ such that $\mu(\Lambda(x_n,\e_n'))>1-2^{-n-1}$.
Observe that $$\sigma_{x_n}(x_n+\e_n y_{x_n})=\sigma_{x_n}(x_n)+\e_n\sigma_{x_n}(y_{x_n})>\max \sigma_{x_n}(\bar C)+\e_n$$and hence $x_n+\e_n y_{x_n}\notin \bar C$.
This completes the inductive step.
\smallskip

The conditions (3) and (4) of the inductive construction guarantee that $A=\{x_n+\e_ny_{x_n}\}_{n\in\w}$ is a required  infinite set, hidden behind the convex set $\bar C$.
\end{proof}

\begin{lemma}\label{l2} If the image $\sigma(G)$ is countable, then the set $\bar C$ hides some infinite subset of $X$.
\end{lemma}

\begin{proof} Denote by $F$ the set of functionals $f\in\sigma(G)$ for which  the set $f^{-1}(\sup f(C))\cap C$ has non-empty interior in $\partial C$.

\begin{claim}\label{cl2.9} The set $F$ is infinite.
\end{claim}

\begin{proof} Assume that the set $F$ is finite and write $F=\{f_1,\dots,f_n\}$ for some functionals $f_1,\dots,f_n\in S^*$. Since $\bar C$ is not polyhedral, $$\bar C\ne\bigcap_{i=1}^nf_i^{-1}((-\infty,\max f_i(\bar C)]).$$ Repeating the argument from Claim~\ref{cl1}, we can find a point $y\in\partial C$ such that $f_i(y)<\max f_i(\bar C)$ for all $i\le n$. Then $U=\bigcap_{i=1}^nf_i^{-1}((-\infty,\max f_i(\bar C))$ is an open neighborhood of $y$ in $X$. Since $G\cap U\subset\bigcup_{f\in \sigma(G)}f^{-1}(\max f(\bar C))$, the Baire Theorem guarantees that for some functional $f\in \sigma(G)$ the intersection $f^{-1}(\max f(\bar C))\cap G\cap U$ has non-empty interior in $G\cap U$. Since $G\cap U$ is dense in $U\cap\partial C$, the intersection $f^{-1}(\max f(\bar C))\cap U$ has non-empty interior in $U\cap\partial C$ and in $\partial C$. Consequently, $f\in F$. Since $f^{-1}(\max f(\bar C))\cap U\ne \emptyset$, we conclude that $f\in F\setminus\{f_1,\dots,f_n\}$, which is a desired contradiction.
\end{proof}

By Claim~\ref{cl2.9}, the set $F\subset \sigma(G)\subset S^*$ is infinite and hence contains an infinite discrete subspace $\{f_n\}_{n\in\w}$. By the definition of $F$, for every $n\in\w$ we can choose a point $x_n\in\partial C$ and a positive real number $\e_n$ such that $\partial C\cap \bar B(x_n,\e_n)\subset f_n^{-1}(\max f_n(\bar C))$. Here $\bar B(x_n,\e_n)=\{x\in X:\|x-x_n\|\le \e\}$ denotes the closed $\e_n$-ball centered at $x_n$. Moreover, since the subspace $\{f_n\}_{n\in\w}$ of $S^*$ is discrete, we can additionally assume that $\bar B(f_n,\e_n)\cap \bar B(f_m,\e_m)=\emptyset$ for any distinct $n,m\in\w$. 
For every $n\in\w$ let $y_n\in S$ be the unique point such that $f_n(z)=\la z,y_n\ra$ for all $z\in X$. The Riesz Representation Theorem guarantees that
$$
\|y_n-y_m\|=\|f_n-f_m\|\ge \e_n+\e_m\mbox{ \ for all $n\ne m$}.
$$

We shall need the following elementary (but not trivial) geometric fact.

\begin{claim}\label{cl4} For any distinct numbers $n,m\in\w$ and a positive real number $\delta_n\le \frac13\e_n^2$ the segment $[x_n+\delta_ny_n,x_m]$ meets the open convex set $C$.
\end{claim}

\begin{proof} Assume conversely that $[x_n+\delta_n y_n,x_m]\cap C=\emptyset$. Taking into account that $f^{-1}_n(f_n(x_n))\cap \bar B(x_n,\e_n)\subset\bar C$, we conclude that $\|x_n-x_m\|\ge\e_n$. Now consider the unit vector $$\mathbf i=\frac{x_m-x_n}{\|x_m-x_n\|}.$$

Since $\la x_m,y_n\ra=f_n(x_m)\le \max f_n(\bar C)=f_n(x_n)=\la x_n,y_n\ra$, we get $\la x_m-x_n,y_n\ra\le 0$, which means that the angle between the vectors $y_n$ and $\mathbf i$ is obtuse. Since $[x_n+\delta_n y_n,x_m]\cap C=\emptyset$, the unit vector $y_n$ is not equal to $-\mathbf i$ and hence the unit vector $$\mathbf j=\frac{y_n-\la\mathbf i,y_n\ra\cdot \mathbf i}{\|y_n-\la\mathbf i,y_n\ra\cdot\mathbf i\|}$$ is well-defined.
Let $\alpha$ be the angle between the vectors $y_n$ and $\mathbf j$. It follows that $y_n=-\sin(\alpha)\,\mathbf i+\cos(\alpha)\,\mathbf j$. Consider the vector $y_n^\perp=\cos(\alpha)\,\mathbf i+\sin(\alpha)\,\mathbf j$, which is  orthogonal to the vector $y_n$. Looking at the following picture, we can see that the angle $\alpha$ is less than the angle $\beta$ between the vectors $y_n^\perp$ and $x_m-(x_n+\delta_ny_n)$.

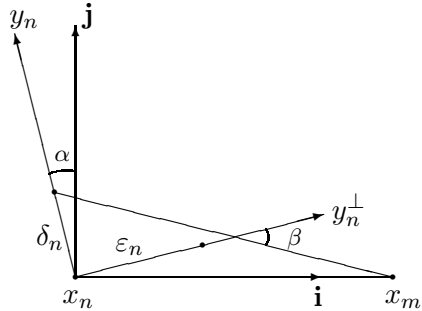
\begin{picture}(200,150)(-120,-20)
\put(0,0){\line(1,0){120}}
\put(0,0){\vector(1,0){93}}
\put(90,-11){$\mathbf i$}
\put(0,0){\circle*{2}}
\put(120,0){\circle*{2}}
\put(117,-10){$x_m$}
\put(0,0){\vector(-1,4){23}}
\put(-25,96){$y_n$}
\put(-15,12){$\delta_n$}
\put(-5,-10){$x_n$}
\put(0,0){\vector(4,1){94}}
\put(97,22){$y_n^\perp$}
\put(15,10){$\e_n$}
\put(48,12){\circle*{2}}
\qbezier(0,40)(-5,41)(-10,39)
\put(-8,44){\footnotesize $\alpha$}
\qbezier(72,18)(74,15)(72,12)
\put(80,12){\footnotesize $\beta$}
\put(0,0){\vector(0,1){95}}
\put(3,96){$\mathbf j$}
\put(-8,32){\circle*{2}}
\put(-8,32){\line(4,-1){128}}
\end{picture}

Since $x_n+\e_n y_n^\perp\in f^{-1}_n(f_n(x_n))\cap \bar B(x_n,\e_n)\subset\bar C$ and $[x_n+\delta_ny_n,x_m]\cap C=\emptyset$, the angle $\beta$ is less than $\arctan(\frac{\delta_n}{\e_n})$. Then
$$\|y_n-\mathbf j\|=2\sin(\alpha/2)\le\alpha\le\beta\le\arctan(\delta_n/\e_n)\le\frac{\delta_n}{\e_n}\le\frac13\e_n.$$

Next, we evaluate the distance $\|y_m-\mathbf j\|$. It is clear that $\|y_m-\mathbf j\|=2\sin(\gamma/2)$ where $\gamma$ is the angle between the vectors $y_m$ and $\mathbf j$.

Let us consider separately two possible cases.
\smallskip

1) The vector $y_m$ lies in the plane spanned by the vectors $\mathbf i$ and $\mathbf j$.
Since $f_n(x)=\la x,y_m\ra$ is a supporting functional for $C$ at the point $x_m$, we get $\la x_n-x_m,y_m\ra\le 0$ and hence $\la \mathbf i,y_m\ra\ge 0$.

On the other hand, $[x_n+\delta_n y_n,x_m]\cap C=\emptyset$ and $f_m^{-1}(f_m(x_m))\cap B(x_m,\e_m)\subset\bar C$ imply that $\la x_n+\delta_n y_n-x_m,y_m\ra\ge 0$ and $\la x_n+\delta_n\mathbf j-x_m,y_m\ra\ge 0$. Consequently, $\gamma<\pi/2$ and $y_m=\sin(\gamma)\mathbf i+\cos(\gamma)\mathbf j$.
It follows that
$$-\|x_n-x_m\|\sin(\gamma)+\delta_n\cos(\gamma)=\la x_n-x_m+\delta_n\mathbf j,y_m\ra\ge0$$ and hence $\tan(\gamma)\le \frac{\delta_n}{\|x_n-x_m\|}\le\frac{\delta_n}{\e_n}$ and
$$\|y_m-\mathbf j\|=2\sin(\gamma/2)\le\gamma\le\tan(\gamma)\le\frac{\delta_n}{\e_n}\le\frac13\e_n.$$
Then $$\|f_n-f_m\|=\|y_n-y_m\|\le\|y_n-\mathbf j\|+\|\mathbf j-y_m\|\le\frac23\e_n<\e_n+\e_m,$$
which contradicts the choice of the sequence $(\e_k)$.
\smallskip

2) The vectors $\mathbf i,\mathbf j, y_m$ are linearly
independent. Let $\mathbf k$ be a vector of unit length in $X$ such
that $\mathbf k$ is orthogonal to $\mathbf i$ and $\mathbf j$ and
$y_m=a\mathbf i+b\mathbf j+c\mathbf k$  for some real numbers
$a,b,c$. It follows that $x_n\pm \e_n\mathbf k\in
f_n^{-1}(f_n(x_n))\cap \bar B(x_n,\e_n)\subset\bar C$. Since $f_m$
is a supporting functional to $C$ at the point $x_m$, we get
$0\ge\la x_n\pm\e_n\mathbf k-x_m,y_m\ra=-\|x_n-x_m\|a\pm\e_nc$,
which implies $$|c|\le
 \frac{\|x_n-x_m\|}{\e_n}a.$$ On the other hand, $[x_n+\delta_n\mathbf j,x_m]\cap C=\emptyset$ implies $0\le\la x_n+\delta_n\mathbf j-x_m,y_m\ra=-\|x_n-x_m\|a+\delta_nb$ and $$\frac{a}{b}\le\frac{\delta_n}{\|x_n-x_m\|}.$$ Now we see that
$$
\begin{aligned}
\|y_m-\mathbf j\|&=2\sin(\gamma/2)\le\tan(\gamma)=\frac{\sqrt{a^2+c^2}}{|b|}\le
\frac{a}{b}\sqrt{1+\frac{\|x_n-x_m\|^2}{\e_n^2}}\le\\
&\le\frac{\delta_n}{\|x_n-x_m\|}\sqrt{1+\frac{\|x_n-x_m\|^2}{\e_n^2}}\le
{\delta_n}\sqrt{\frac1{\|x_n-x_m\|^2}+\frac{1}{\e_n^2}}\le\\
&\le\delta_n\sqrt{\frac1{\e_n^2}+\frac1{\e_n^2}}=\sqrt{2}\frac{\delta_n}{\e_n}\le\frac{\sqrt{2}}3\e_n.
\end{aligned}
$$
Then $\|f_n-f_m\|=\|y_n-y_m\|\le\|y_n-\mathbf j\|+\|\mathbf j-y_m\|\le\frac13\e_n+\frac{\sqrt{2}}3\e_n<\e_n+\e_m$, which contradicts the choice of the sequence $(\e_k)$. This contradiction completes the proof of Claim~\ref{cl4}.
\end{proof}

Now we can continue to prove Lemma~\ref{l2}.
By induction for every $n\in\w$ we shall choose a positive real number $\delta_{n}$ such that
\begin{enumerate}
\item $\delta_n\le \frac13\e_n^2$;
\item $[x_k+\delta_ky_k,x_n+\delta_ny_n]\cap C\ne\emptyset$ for any $k<n$.
\end{enumerate}

To start the inductive construction put $\delta_0=\frac13\e_0^2$. Assume that for some $n\in\w$ we have constructed positive real numbers $\delta_k$, $k<n$, satisfying the conditions (1)--(2). By Claim~\ref{cl4}, for every $k<n$ the intersection $[x_k+\delta_k y_k,x_n]\cap C$ is not empty. Since the set $C$ is open, we can choose a positive $\delta_n\le\frac13\e_n^2$ so small that for every $k<n$ the intersection $[x_k+\delta_ky_k,x_n+\delta_ny_n]\cap C$ still is not empty. This completes the inductive construction.

It follows from (2) that the infinite set $A=\{x_n+\delta_ny_n\}_{n\in\w}$ is hidden behind the convex set $\bar C$.
\end{proof}
Lemmas~\ref{l2.6} and \ref{l2} complete the proof of Lemma~\ref{l0}.
\end{proof}

\section{Proof of Theorem~\ref{main}}\label{main:pf}

The implications  $(1)\Ra(2)\Ra(3)\Ra(1)$ of Theorem~\ref{main} are proved in the following three lemmas.

\begin{lemma} If $C$ is polyhedral in $\overline{\lin}(C)$, then $C$ is polyhedric in $\aff(C)$.
\end{lemma}

\begin{proof} If $C$ is polyhedral in $\overline{\lin}(C)$, then $C=\bigcap_{i=1}^n f_i^{-1}((-\infty,a_i])$ for some linear functionals $f_1,\dots,f_n:\overline{\lin}(C)\to\IR$ and some real numbers $a_1,\dots,a_n$.   For every $i\le n$ consider the convex set $H_i=\aff(C)\cap f_i^{-1}((-\infty,a_i])$ and observe that its complement $\aff(C)\setminus H_i=\aff(C)\cap f_i^{-1}((a_i,+\infty))$ also is convex. Since $C=\bigcap_{i=1}^n H_i$, the set $C$ is polyhedric in $\aff(C)$.
\end{proof}

\begin{lemma} If a convex subset $C$ of a linear space $X$ is polyhedric in $\aff(C)$, then $C$ hides no infinite subset $A\subset X\setminus C$.
\end{lemma}

\begin{proof} Assume conversely that some infinite subset $A\subset X\setminus C$ can be hidden behind $C$. First we show that the set $A\setminus \aff(C)$ contains at most two distinct points. Assume conversely that there are three pairwise distinct points $a_1,a_2,a_3\in A\setminus \aff(C)$. Let $P=\{t_1a_1+t_2a_2+t_3a_3:t_1+t_2+t_3=1\}$ be the affine subspace of $X$ spanned by the points $a_1,a_2,a_3$. The subspace $P$ has dimension 1 or 2. The intersection $P\cap \aff(C)$ is an affine subspace of $P$ that intersects the open segments $]a_1,a_2[$, $]a_1,a_3[$ and $]a_2,a_3[$ and hence coincides with $P$, which is not possible as $a_1,a_2,a_3\in P\setminus \aff(C)$.
So, $|A\setminus\aff(C)|\le 2$ and we lose no generality assuming that $A\subset\aff(C)$.

Being polyhedric in $\aff(C)$, the set $C$ can be written as a finite intersection $C=\bigcap_{i=1}^n H_i$ of convex subsets $H_1,\dots,H_n\subset \aff(C)$ having convex complements $\aff(C)\setminus H_i$, $i\le n$. Since $A\setminus C=\bigcup_{i=1}^n\aff(C)\setminus H_i$, by the Pigeonhole Principle, there is an index $i\in\{1,\dots,n\}$ such that the convex set $\aff(C)\setminus H_i$ contains  two distinct points $a,b\in A$ and hence contains the segment $[a,b]$, which is not possible as $[a,b]$ meets the set $C\subset H_i$.
\end{proof}

\begin{lemma} If a closed convex subset $C$ of a complete metric linear space $X$ is not polyhedral in $\overline{\lin}(C)$, then some infinite  $A\subset X\setminus C$ can be hidden behind $C$.
\end{lemma}

\begin{proof} Assume that $C$ is not polyhedral in $\overline{\lin}(C)$. It is easy to check that $$\Ker(C)=\{x\in X:\forall c\in C\;\forall t\in\IR\;\;c+tx\in C\}$$ is a closed linear subspace of $X$ and $C=C+\Ker(C)$. Let $Y=X/\Ker(C)$ be the quotient linear metric space and $Q:X\to Y$ be the quotient operator. By \cite[2.3.1]{Rol}, the operator $Q$ is open and by \cite[1.4.10]{Rol}, $Y$ is a complete linear metric space. Let $D=Q(C)$. The equality $C=C+\Ker(C)$ implies that $C=Q^{-1}(D)$ and $Y\setminus D=Q(X\setminus C)$ is an open set. So, $D$ is a closed convex set in $Y$. By Lemma~\ref{l2.1}, the set $D$ is not polyhedral in its closed linear hull $\overline{\lin}(D)$.

If the linear space $\overline{\lin}(D)$ is finite-dimensional, then it is isomorphic to a finite-dimensional Hilbert space $H$. Let $T:H\to\overline{\lin}(D)$ be the corresponding isomorphism. Since $D$ is not polyhedral in $\overline{\lin}(D)$, the preimage $E=T^{-1}(D)$ is not polyhedral in the Hilbert space $H$. Being finite-dimensional, the closed convex set $E$ is a convex body in its affine hull $\aff(E)\subset H$. Then for every $e_0\in E$ the convex set $E_0=E-e_0$ is a convex body in the linear subspace $H_0=\aff(E)-e_0$ of $H$. Since $E$ is not polyhedral in $H$, the shift $E_0=E-e_0$ is not polyhedral in the Hilbert space $H_0$. By Lemma~\ref{l0} the set $E_0$ hides an infinite subset $A_0\subset H_0\setminus E_0$. Then the set $E$ hides the infinite set $A_0+e_0$ and the set $T(E)=D$ hides the infinite set $B=T(A_0+e_0)$. Choose any subset $A\subset X$ such that $Q|A:A\to B$ is bijective. By Lemma~\ref{l2.2} the infinite set $A$ is hidden behind the convex set $C=Q^{-1}(D)$ and we are done.

Next, assume that the linear space $\overline{\lin}(D)$ is infinite-dimensional. Then the convex set $D$ also is infinite-dimensional. By Lemma~\ref{l2.3}, there is a continuous injective affine operator $T:l_2\to \overline{\lin}(D)$ such that $E=T^{-1}(D)$ is a closed convex body in $l_2$. Since $\Ker(D)=\{0\}$, we get $\Ker(E)=\{0\}$. This implies that the set $E$ is not polyhedral in the Hilbert $l_2$. By Lemma~\ref{l0}, the convex set $E$ hides some infinite subset $A_0\subset l_2\setminus E$. Then the infinite set $B=T(A_0)\subset Y\setminus D$ is hidden behind the convex set $D$. Choose any subset $A\subset X$ such that $Q|A:A\to B$ is bijective. By Lemma~\ref{l2.2} the infinite set $A$ is hidden behind the convex set $C=Q^{-1}(D)$ and we are done.
\end{proof}

\section{Open Problems}

It would be interesting to know whether a relative version of Theorem~\ref{main} is true.

\begin{problem} Let $C\subset D$ be two closed convex subsets of a complete linear metric space. Is it true that $C$ hides no infinite subset $A\subset D\setminus C$ if and only if $C$ is polyhedric in $D\cap\aff(C)$?
\end{problem}

In fact, the notions of polyhedric and hidden sets can be defined in a general context of convex structures, see \cite{Vel}. Let us recall that a {\em convex structure} on a set $X$ is a family $\C$ of subsets of $X$ such that
\begin{itemize}
\item $\emptyset,X\in\mathcal C$;
\item for any subfamily $\A\subset\C$ the intersection $\cap\A\in\C$;
\item for any linearly ordered subfamily $\A\subset\A$ we get $\cup\A\in\C$.
\end{itemize}
For a convex structure $(X,\C)$ and a subset $A\subset C$ the intersection $\conv(A)=\bigcap\{C\in\C:A\subset C\}$ is called the {\em convex hull} of the set $A$.

We say that a subset $C\subset X$ {\em hides} a subset $A\subset X$ if $\conv(\{a,b\})\cap C\ne\emptyset$ for any two distinct points $a,b\in A$.

A subset $C$ is {\em polyhedric in a subset $D\supset C$} if $C=\bigcap_{i=1}^nH_i$ for some subsets $H_1,\dots,H_n\subset D$ such that $H_i,D\setminus H_i\in \C$ for all $i\le n$.

\begin{problem} Given a convex structure $(X,\C)$  (possibly with topology) characterize (closed) convex sets $C\in\C$ that hide no infinite subset $A\subset X\setminus C$.
\end{problem}

\end{document}